\documentclass[a4paper,11pt]{article}
\usepackage[T1]{fontenc}
\usepackage{amsmath, amsfonts, amsthm}
\usepackage{xcolor}
\usepackage{bbm}
\usepackage{graphicx}
\usepackage[a4paper]{geometry}
\usepackage{dsfont}
\usepackage{hyperref}
\hypersetup{
colorlinks=true,
linkcolor=blue,
citecolor=red
}

\newtheorem{theorem}{Theorem}[section]

\newtheorem{cor}{Corollary}[section]

\theoremstyle{definition}

\newtheorem{remark}{Remark}[section]

\theoremstyle{proof}

\newcommand{\E}{\mathbb{E}}

\newcommand{\R}{\mathbb{R}}
\newcommand{\N}{\mathbb{N}}

\newcommand{\pr}[1]{\mathbb{P}\left( #1 \right)}

\newcommand{\dra}{\stackrel{d}{\longrightarrow}}

\newcommand{\normal}[1]{\text{\normalfont{#1}}}

\newcommand{\ra}{\rightarrow}

\newcommand{\deq}{\stackrel{d}{=}}

\newcommand{\bround}[1]{\Big( #1 \Big)}

\newcommand{\bsq}[1]{\Big[ #1 \Big]}
\newcommand{\br}[1]{\left\{ #1 \right\}}

\newcommand{\ddd}{,\dots,}

\newcommand{\head}[1]{\vspace{.5cm}\noindent \textbf{#1}}

\title{\textbf{Formal construction of some exchangeable structures}}
\author{Minh-Toan Nguyen \\ GIPSA-lab, Grenoble Alpes University}
\date{}

\begin{document}
\maketitle
\begin{abstract}
We show that exchangeable structures such as Polya urn model and Chinese restaurant process can be constructed from sets with a real number of elements. From this construction, the exchangeability of these structures becomes obvious and the calculations on them become extremely simple.
\end{abstract}

\section{Introduction}
The cardinality of any set is a natural number. We will break this rule by considering a general kind of sets, called \emph{formal sets}, which have a real number of elements. We will not try to make sense of the formal sets, instead we will treat them like usual sets and build from them meaningful structures such as Polya urn model and Chinese restaurant process. From this viewpoint, we can see through some non-trivial properties of these structures and the related objects. One of these properties is exchangeability, satisfied when a random structure has its law unchanged when a finite number of its elements are permuted. The exchangeability, which comes as a surprise from the definition of these structures, becomes obvious in the formal constructions. Moreover, the formal construction greatly simplify calculations. The usual calculations on these structures, which involve induction, integrals and Jacobian determinants, now can be done by simple combinatorial calculations.

The Polya urn model and the Chinese restaurant process that we will build from the formal sets are important probabilistic structures. The Polya urn model is a classic example of infinite exchangeable  sequences. It is closely related to Dirichlet distribution \cite{lin2016dirichlet} frequently used in Bayesian statistics. The Chinese restaurant process is an exchangeable random partition of $\N$. It gives rise to Ewens sampling formula \cite{crane2016ubiquitous}  and Poisson-Dirichlet distribution, which appears in various mathematical problems such as random walk \cite{derrida1997random}, Brownian motion \cite{pitman1997two}, fragmentation and coalescent process \cite{berestycki2009recent} \cite{bertoin2006random} and prime factorization of large integers \cite{billingsley1972distribution} \cite{donnelly1993asymptotic}. The Polya urn model and the Chinese restaurant process are members of a larger family of exchangeable structures widely used in Bayesian topic model, as they provide a flexible and elegant framework for modeling data without assuming a fixed number of clusters. Some notable members of this family are Latent Dirichlet model \cite{blei2003latent}  \cite{pritchard2000inference}, Indian buffet process \cite{griffiths2011indian}, hierarchical Dirichlet process \cite{teh2004sharing} and nested Chinese restaurant process \cite{blei2010nested} \cite{griffiths2003hierarchical}.

\section{Exchangeability}\label{exchange}
A finite sequence $(Y_1 \ddd Y_n)$ of random variables is \emph{exchangeable} if
\begin{align}
(Y_1 \ddd Y_n) \deq (Y_{\sigma(1)} \ddd Y_{\sigma(n)})
\end{align}
for each permutation $\sigma$ of $[n]$. An infinite sequence $(Y_i)_{i=1}^\infty$ is exchangeable if 
\begin{align}
(Y_1,Y_2,\dots) \deq (Y_{\sigma(1)}, Y_{\sigma(2)}, \dots)
\end{align}
for each finite permutation of $\N_+$, i.e. permutations such that $\br{i: \sigma(i)\neq i}$ is finite.

Exchangeability arises naturally from sampling. Consider a set with $m$ elements, each with a label that is not necessarily unique. Then if we randomly draw $n<m$ elements from the set without replacement, the labels of these elements form an exchangeable sequence. More generally, consider a random vector $(X_1, \ldots, X_m)$. If we select $n$ distinct indices $i_1, \ldots, i_n$ randomly from the set $[m]$, the resulting sequence $(X_{i_1}, \ldots, X_{i_n})$ is exchangeable.

We can create an exchangeable sequence by picking a random probability measure and draw an i.i.d. sequence from it. De Finetti's theorem states that all infinite exchangeable sequence can be constructed this way. In other words, 
\begin{theorem}
Every infinite exchangeable sequence is a mixture of i.i.d. sequences.
\end{theorem}
Consider an exchangeable sequence $(Y_i)_{i=1}^{\infty}$ where each $Y_i$ takes values in a discrete set $\mathcal Y$. By the law of large number and by de Finetti's theorem, each element of $\mathcal Y$ has a limiting proportion in  $(Y_i)_{i=1}^\infty$. Moreover, these proportions vary for different realizations of $(Y_i)_{i=1}^\infty$.

\section{Polya urn model}\label{polya}
\subsection{The model}
In this model, at the beginning we have a set containing elements labeled by $1 \ddd k$. Let $\alpha_i \in \N_+$ be the number of elements with label $i$. At each step, we choose uniformly randomly a element from the set, record its label, and put it back along with another element of the same label. Let $(Y_i)_{i=1}^{\infty}$ be the sequence of the recorded labels. This sequence is a stochastic process that satisfies
\begin{align}\label{polya1}
\pr{Y_1=i} = \frac{\alpha_i}{\alpha_1 + \dots + \alpha_k}
\end{align}
and
\begin{align}\label{polya2}
\pr{Y_{n+1}=i|Y_1 \ddd Y_n} = \frac{\alpha_i + n_i}{\alpha_1+\dots+\alpha_k + n},
\end{align}
where $n_i$ is the number of occurrence of $i$ in the sequence $Y_1 \ddd Y_n$. The equations (\ref{polya1}) and (\ref{polya2}) uniquely determines a stochastic process for $\alpha_1 \ddd \alpha_k$ that are not restricted to $\N_+$, but rather extend to $\R_+$. We call such process \emph{Polya urn process} with parameters $(\alpha_1 \ddd \alpha_k) \in \R^k_+$.

The Polya urn process has the remarkable property of being exchangeable, which is not at all trivial from the definition, since to prove it one has no other way than doing explicit calculations.

\subsection{Formal set construction}
Imagine a set containing a total of $-\alpha_1-\dots-\alpha_k$ elements divided into $k$ groups labeled from $1$ to $k$, each with sizes $-\alpha_1 \ddd -\alpha_k$ respectively, where $\alpha_1 \ddd \alpha_k > 0$. Forgetting the fact that the cardinality of a set must be a non-negative integer, let us see what happens when we sample without replacement from this set. The probability that the first element has label $i$ is
\begin{align*}
\frac{-\alpha_i}{-\alpha_1-\dots-\alpha_k} = \frac{\alpha_i}{\alpha_1+\dots+\alpha_k},
\end{align*}
Suppose that after $n$ steps,  we have taken out $n_i$ elements with label $i$. In the set there remains $-\alpha_i-n_i$ elements with label $i$. The probability of the $(n+1)$-th element having label $i$ is
\begin{align*}
\frac{-\alpha_i - n_i}{-\alpha_1-n_1-\dots -\alpha_k-n_k} =\frac{\alpha_i+n_i}{\alpha_1+\dots+\alpha_k+n},
\end{align*}
Although the underlying set is ill-defined, the probabilities arising from the sampling process are well-defined and precisely match those of the Polya urn model. With formal sets, the Polya urn model is described more concisely and its exchangeability becomes trivial.

\subsection{Joint probability}
Let $(y_1 \ddd y_n)$ be the sequence of labels in the first $n$ samplings. Suppose that the label $i$ appears $n_i$ times in this sequence. From the formal set construction we have

\begin{samepage}
\begin{align*}
\pr{Y_1=y_1 \ddd Y_n=y_n} &= \frac{(-\alpha_1)^{\downarrow n_1}\dots (-\alpha_k)^{\downarrow n_k}}{(-\alpha_1 - \dots-\alpha_k)^{\downarrow n }}
\end{align*}
Here the denominator counts the sequences of $n$ different elements from the formal set and $(-\alpha_i)^{\downarrow n_i}$ counts the sequences of different $n_i$ elements with label $i$. From the formula $(-x)^{\downarrow n } = (-1)^n x^{\uparrow n}$, we obtain
\begin{align}\label{tr}
\pr{Y_1=y_1 \ddd Y_n=y_n} &= \frac{ \alpha_1^{\uparrow n_1} \dots \alpha_k^{\uparrow n_k} }{( \alpha_1 + \dots + \alpha_k )^{\uparrow n}}
\end{align}
\end{samepage}

Note that with the usual definition of Polya urn model, we prove (\ref{tr}) by induction and then conclude that the sequence $(Y_i)_{i=1}^n$ is exchangeable.

\subsection{Dirichlet distribution}
Consider a Polya urn sequence $(Y_i)_{i=1}^n$ with parameters $(\alpha_1 \ddd \alpha_k)$. The probability of having $n_i$ labels $i$ in the sequence $(Y_1 \ddd Y_n)$ is 
\begin{align*}
p(n_1 \ddd n_k) = \frac{n!}{n_1! \dots n_k!} .\frac{ \alpha_1^{\uparrow n_1} \dots \alpha_k^{\uparrow n_k} }{( \alpha_1 + \dots + \alpha_k )^{\uparrow n}}
\end{align*}
Let $x_i = n_i/n$. Using the fact that
\begin{align}\label{yo}
\frac{\Gamma(m+r)}{\Gamma(m+s)} \simeq m^{r-s}, \quad m \ra \infty,
\end{align}
we obtain the density of $(x_1 \ddd x_k)$ as $n$ tends to infinity:
\begin{align*}
f(x_1 \ddd x_k) = \frac{\Gamma(\alpha_1+\dots+\alpha_k)}{\Gamma(\alpha_1)\dots\Gamma(\alpha_k)} x_1^{\alpha_1-1} \dots x_k^{\alpha_k-1} 1_{\Delta_k}(x_1 \ddd x_k)
\end{align*}
where
\begin{align*}
\Delta_k = \br{x_1 \ddd x_k \geq 0: x_1+\dots+x_k=1}.
\end{align*}
This is the Dirichlet distribution $\text{Dir}(\alpha_1 \ddd \alpha_n)$. In summary, the proportions of $1 \ddd k$ in an infinite Polya urn sequence with parameters $(\alpha_1 \ddd \alpha_k)$ follow the Dirichlet distribution with the same parameters.

From the formal set construction of Polya urn model, the following properties of the Dirichlet distribution can be derived with very little effort.

\begin{theorem}\label{dir1}
Consider a random vector $(X_1, \ldots, X_n)$ following a Dirichlet distribution with parameters $(\alpha_1, \ldots, \alpha_n)$. Then
\begin{enumerate}
\item \label{dir11} \normal{(Aggregation)} If $[n]$ is partitioned into subsets $B_1 \ddd B_k$, then
\begin{align}
\bround{\sum_{i \in B_1} X_i \ddd \sum_{i \in B_k} X_i} \sim \normal{Dir} \bround{\sum_{i \in B_1}\alpha_i \ddd \sum_{i \in B_k} \alpha_i}
\end{align}

\item \label{dir12} \normal{(Neutrality)} Let $I$ be an ordered subset\footnote{An ordered set is simply a set with elements arranged in some order. For $I=(i_1 \ddd i_k)$, denote $x_I = (x_{i_1} \ddd x_{i_k})$ .}  of $[n]$, and $\tilde X_I = X_I/\sum_{i \in I} X_i$. Then $\tilde X_I$ follows $ \normal{Dir}(\alpha_{I})$ and is independent of $X_{I^c}$.
\end{enumerate}
\end{theorem}

\begin{proof}
$(X_1 \ddd X_n)$ can be viewed as the proportion of $1\ddd n$ in a Polya urn sequence $Y=(Y_i)_{i=1}^\infty$ with parameters $(\alpha_1 \ddd \alpha_n)$. Consider the formal set construction of this Polya urn model.

The property \ref{dir11} can be easily proved by combining all the elements labeled by $B_j$ in the formal set into a new group for each $j=1\ddd k$.

To prove \ref{dir12}, imagine assigning a special label $0$ to each element in the formal set that already has a label from $I$. This new label temporarily conceals the original one. Next, we sample without replacement from the formal set as usual. Finally, we remove the labels $0$ to reveal the original labels. The sequence of revealed labels forms a Polya urn process with parameters $\alpha_I$. This sequence is independent of the observations before the revealing, therefore independent of $X_{I^c}$.
\end{proof}

\begin{cor}\label{dir2}
Let $(X_1 \ddd X_n)$ be a Dirichlet random vector with parameters $\alpha_1 \ddd \alpha_n > 0$. Then
\begin{enumerate}
\item \label{dir21} \normal{(Marginal law)} 
\begin{align*}
P_{X_1 \ddd X_k}(x_1 \ddd x_k) \propto x_1^{\alpha_1-1}\dots x_k^{\alpha_k-1}(1-x_1 - \dots - x_k)^{\alpha_{k+1}+\dots+\alpha_n-1}
\end{align*}
In particular
\begin{align}
X_i \sim \normal{Beta}\bround{ \alpha_i, \sum_{j \neq i} \alpha_j }
\end{align}

\item \label{dir22} \normal{(Gamma construction)} Let $Z_i \sim \Gamma(\alpha_i, 1)$ be independent. Then
\begin{align}
(X_1 \ddd X_n) \deq \bround{ \frac{Z_i}{Z_1+\dots+Z_n} }_{i=1}^n 
\end{align}
\item \label{dir23} \normal{(Stick-breaking construction)} Consider $(X_1^* \ddd X_n^*)$ constructed as follows
\begin{align*}
X_1^* &= W_1 \\
X_2^* &= W_2(1-W_1)  \\
\dots \\
X_{n-1}^* &= W_{n-1}(1-W_{n-2}) \dots (1-W_1)
\end{align*}
and $X_n^* = 1-X_1^*-\dots-X_n^*$, where $W_i$ are independent with
\begin{align*}
W_i \sim \normal{Beta}(\alpha_i, \alpha_{i+1}+\dots +\alpha_n)
\end{align*}
Then
\begin{align*}
(X_1 \ddd X_n) \deq (X_1^* \ddd X_n^*)
\end{align*}
\end{enumerate}
\end{cor}

\begin{proof}
The claim \ref{dir21} follows from
\begin{align*}
(X_1 \ddd X_k, X_{k+1}+\dots+X_n) \sim \text{Dir}(\alpha_1 \ddd \alpha_k, \alpha_{k+1}+\dots+\alpha_n).
\end{align*}
To prove the claim \ref{dir22}, consider
\begin{align*}
(V_1 \ddd V_n, T) \sim \text{Dir}(\alpha_1 \ddd \alpha_n, t)
\end{align*}
It is easy to check that
\begin{align*}
(tV_1 \ddd tV_n) \dra (Z_1 \ddd Z_n), \quad t \ra \infty
\end{align*}
On the other hand, from Theorem \ref{dir1}, we have
\begin{align*}
\bround{ \frac{tV_i}{tV_1+\dots+tV_n} }_{i=1}^n \deq (X_1 \ddd X_n)
\end{align*}
By taking $t \ra \infty$, we obtain the claim \ref{dir22}. Claim \ref{dir23} follows from the neutrality of Dirichlet distribution. First we have $X_1 \deq W_1$. Then the relative size of $X_2$ in $(X_2 \ddd X_n)$ is independent of $X_1$ and has the same law as $W_2$, so $X_2 \deq W_2(1-W_1)$. The claim can be proved by repeating this argument.
\end{proof}

\section{Chinese restaurant process}\label{chinese}
\subsection{The model}
Let $\alpha, \theta$ be parameters that satisfies either one of the following cases
\begin{enumerate}
\item $\alpha<0$ and $\theta=-k\alpha$ for some $k \in \N_+$ or \label{crp1}
\item $0\leq \alpha \leq 1$ and $\theta > -\alpha$. \label{crp2}
\end{enumerate}
The Chinese restaurant process with parameters $(\alpha, \theta)$ is defined as follows. Imagine a restaurant with an infinite number of tables. At the beginning all tables are empty. The first customer arrives and sits at any table. If the first $n$ customers occupy $m$ tables, then the $(n+1)$-th customer will choose
\begin{itemize}
\item a table with $t \geq 1$ customers with probability $\frac{t-\alpha}{n+\theta}$
\item an empty table with probability $\frac{m\alpha + \theta}{n + \theta}$
\end{itemize}
The first $n$ customers form a partition of $[n]$ in which tables represent blocks and customers represent elements. By continuing this process infinitely we obtain a random partition of $\N$. This random partition has the remarkable property of exchangeability, meaning that its law remains unchanged under finite permutations of $\N$.

\subsection{Formal set construction}
Consider a set of $-\theta$ elements divided into $-\theta/\alpha$ groups, each containing $\alpha$ elements. Let us see what happens when we sample without replacement from this set. Suppose the first $n$ sampled elements belong to $m$ groups. After this, there are $\theta - n$ elements remaining in the set. For a group with $t$ sampled elements, there are $\alpha - t$ elements of that group still in the set. Therefore, the probability of the $(n+1)$-th element coming from this group is:
\begin{align*}
\frac{\alpha - t}{-\theta - n} = \frac{t-\alpha}{n+\theta}.
\end{align*}
Since all the probabilities add up to $1$, the probability for the $(n+1)$-th element to be in a new group is 
\begin{align*}
\frac{m\alpha + \theta}{n + \theta}.
\end{align*}
These probabilities exactly match those of Chinese restaurant process. With formal sets, the Chinese restaurant process can be described more concisely and it becomes trivial that the resulting partition of $\N$ is exchangeable.

For the parameters $(\alpha, \theta)$ in case \ref{crp1}, the formal set consists of $k$ blocks of size $\alpha$, where $k\in\N_+$. Thus, the Chinese restaurant process is the same as the Polya urn process with $k$ labels and parameters $(-\alpha \ddd -\alpha)$. The resulting random partition on $\N$ always has $k$ blocks. In contrast, for the parameters $(\alpha, \theta)$ in case \ref{crp2}, we will see that the resulting partition of $\N$ contains an infinite number of blocks almost surely.

\subsection{Ewens-Pitman distribution}
\begin{theorem}\label{ewens} \normal{(Ewens-Pitman distribution)}
Consider the Chinese restaurant process with parameters $(\alpha, \theta)$. Let $\Pi_n$ be the random partition of $[n]$ formed by the first $n$ customers. Let $\pi$ be a partition of $[n]$ with block sizes $n_1 \ddd n_k$. Then
\begin{align}\label{ewens-1}
\pr{\Pi_n=\pi} = \frac{(\theta/\alpha)^{\uparrow k}}{\theta^{\uparrow n}} \prod_{i=1}^k -(-\alpha)^{\uparrow n_i}
\end{align}
\end{theorem}
This surely can be proved by induction, starting from the definition, although that would be tedious. We give here a very short derivation using formal sets.

\begin{proof}
Recall that the Chinese restaurant process is equivalent to sampling without replacement from a formal set of $-\theta$ elements divided into $-\theta/\alpha$ groups of size $\alpha$. There are $(-\theta)^{\downarrow n}$ ways of choosing a sequence of $n$ different elements from the formal set. Next, we will compute the number of such sequences that lead to the partition $\pi$. There are $(-\theta/\alpha)^{\downarrow k} $ ways of assigning groups of the formal set to $k$ blocks of $\pi$. Afterwards, for each block of size $n_i$, there are $\alpha^{\downarrow n_i}$ ways to select its elements. Therefore
\begin{align}
\pr{\Pi_n = \pi} = \frac{(-\theta/\alpha)^{\downarrow k}}{(-\theta)^{\downarrow n}} \prod_{i=1}^k \alpha^{\downarrow n_i}
\end{align}
From the fact that $(-x)^{\downarrow n} = (-1)^n x^{\uparrow n}$, we obtain the result given by (\ref{ewens-1}).
\end{proof}

\begin{remark} In Theorem \ref{ewens}, for $\alpha = 0$, we have Ewens sampling formula
\begin{align*}
\pr{\Pi_n=\pi} = \frac{\theta^k}{\theta^{\uparrow n}} \prod_{i=1}^k (n_i-1)!
\end{align*}
For $\theta = 0$, we have
\begin{align*}
\pr{\Pi_n=\pi} =\frac{ (k-1)!}{\alpha (n-1)!} \prod_{i=1}^k -(-\alpha)^{\uparrow n_i} 
\end{align*}
These formulas can be obtained by setting $\alpha \ra 0$ or $\theta \ra 0$ in the Ewens-Pitman distribution.
\end{remark}

\subsection{Block weights}
Consider the random partition of $\N$ generated from the Chinese restaurant process with parameters $(\alpha, \theta)$. Let $B_1$ be the block containing $1$, $B_2$ be the the block containing the smallest element not in $B_1$, $B_3$ be the block containing the smallest element not in $B_1, B_2$, and so on. For each $i$, the \emph{weight} of $B_i$ is defined as
\begin{align}\label{cq}
V_i = \lim_{n \ra \infty} \frac{|B_i \cap [n] |}{n}
\end{align}
It is clear that $\sum_i V_i=1$. We have the following result, called \emph{stick-breaking construction} of $(V_i)$:
\begin{theorem}
The sequence $(V_1, V_2, \dots)$ in (\ref{cq}) is well-defined and has the same law as  $(V_1^\star, V_2^\star, \dots)$, where
\begin{align*}
V_1^\star &= W_1 \\
V_2^\star &= W_2(1-W_1)\\
\dots \\
V_k^\star &= W_k(1-W_{k-1})\dots(1-W_1)\\
\dots
\end{align*}
for $W_1,W_2,\dots$ independent and
\begin{align*}
W_j \sim \normal{Beta}(1-\alpha, \theta+j\alpha)
\end{align*}
for each $j=1,2,\dots$. As a consequence, the random partition has an infinite number of blocks, each with a positive weight.
\end{theorem}

\begin{proof}
Consider the formal set that gives rise to the Chinese restaurant process with parameters $(\alpha, \theta)$. Let us call $x_1$ the first element drawn from the formal set. After $x_1$ is drawn, in the formal set there remain $\alpha - 1$ elements within the same group as $x_1$, and $-\theta - \alpha$ other elements. From our discussion on Polya urn model, in the infinite sequence drawn from this remaining set, the proportion of elements of the same group as $x_1$ is $W_1 \sim \text{Beta}(1-\alpha, \theta+\alpha)$. 

Now let us ignore the block containing $x_1$ and consider the remaining sequence. The blocks in this sequence have a total weight of $1-W_1$ and are generated from the drawing on a formal set of $-\theta -\alpha $ elements divided into groups of size $\alpha$. Let $x_2$ be the first element of this sequence. By the same argument as in the previous paragraph, the block containing $x_2$ has relative size $W_2 \sim \text{Beta}(1-\alpha, \theta + 2\alpha)$, therefore its size is $W_2(1-W_1)$. 

Ignoring blocks containing $x_1, x_2$, the rest has a total weight of $(1-W_1)(1-W_2)$. By repeating the same argument we obtain the result of the theorem.
\end{proof}

By rearranging the sequence $(V_i)$ in decreasing order, we obtain the sequence $(S_i)$. These two sequences are reorderings of each other. Specifically, $(S_i)$ is the decreasing reordering of $(V_i)$, and $V_i$ is referred to as the \emph{size-biased reordering} of $(S_i)$.

The sequence $S_1 \geq S_2 \geq \dots$ is a sample from the \emph{Poisson-Dirichlet distribution} with parameters $(\alpha, \theta)$, which is a probability distribution on the set of non-increasing sequences of positive numbers adding up to $1$. Results on Poisson-Dirichlet distribution can be found in \cite{handa2009two} \cite{perman1992size}  \cite{pitman2006combinatorial} \cite{pitman1997two}. 

While the weights $(V_i)$ are rather simple, as they can be described by the stick-breaking process, the weights $(S_i)$ are much more complicated as the density of $S_i$ for a fixed $i$ can have singular points \cite{derrida1987statistical}.

\head{Correlation functions.} The set $\br{S_i, i \in \N}$ define a point process in $\R$. The $k$-correlation function of a random point process $\br{X_i}$ in $\R$ is defined as a function $\rho_k(x_1 \ddd x_k)$ such that the probability of the process having one point in each of the intervals $[x_i, x_i+dx_i] $ for $i=1\ddd k$ is $\rho_k(x_1 \ddd x_k)dx_1 \dots dx_k$ as $dx_i \ra 0$. Equivalently, the $k$-correlation function can be defined as satisfying
\begin{align}
\E \bsq{ \sum_{i_1 \ddd i_k (\neq)} f(X_{i_1}\ddd X_{i_k}) } = \int f(x_1 \ddd x_k) \rho(x_1 \ddd x_k) dx_1 \dots dx_k
\end{align}
for any non-negative measurable function $f$, where the sum is over $k$ different indices. Since $f \geq 0$, the sum on the left hand side has always a limit in $[0, \infty]$. It follows from both definitions that $\rho_k$ is invariant by permutations of its variables. Note that $\rho_k$ is not a probability density: if $\br{X_i}$ has infinite points then the integral of $\rho_k$ over $x_1 \ddd x_k$ is infinite.

Let us compute the correlation function of the point process $\br{S_i}$. Imagine drawing $n$ elements from a the formal set with $-\theta$ elements divided into groups of size $\alpha$. We will compute the probability $P(n_1 \ddd n_k)$ that there are blocks of sizes $n_1 \ddd n_k$ in the partition formed by these $n$ elements, where $n_1 \ddd n_k \in \N_+$ such that $n_1+\dots+n_k \leq n$. There are
\begin{itemize}
\item $ \binom{-\theta/\alpha}{k}$ ways of choosing $k$ groups in the formal set.
\item $\binom{\alpha}{n_1} \dots \binom{\alpha}{n_k}$ ways of choosing the elements for the blocks of sizes $n_1 \ddd n_k$.
\item $\binom{-\theta-k\alpha}{n-n_1-\dots-n_k}$ ways of choosing other elements.
\item $\binom{-\theta}{n}$ ways of choosing $n$ elements from the formal set.
\end{itemize}
$P(n_1 \ddd n_1)$ is obtained as the product of the first three numbers divided by the fourth number. Let $x_k = n_k/n$. As $n$ tends to infinity, by using (\ref{yo}) and the formula
\begin{align*}
\binom{-x}{m} = \frac{(-1)^m \Gamma(x+m)}{\Gamma(x) \Gamma(m+1)}
\end{align*}
we obtain the $k$-correlation function for the block weights \cite{handa2009two}
\begin{align}
\rho_k(x_1 \ddd x_k) = c_{k,\alpha,\theta} \, x_1^{-\alpha-1} \dots x_k^{-\alpha-1}(1-x_1-\dots - x_k)^{ k\alpha + \theta - 1 }
\end{align}
where
\begin{align}
c_{k,\alpha,\theta} = \frac{ \Gamma(\theta/\alpha+k) \Gamma(\theta) \alpha^k }{ \Gamma(\theta+k\alpha) \Gamma(\theta/\alpha) \Gamma(1-\alpha)^k }.
\end{align}

\bibliographystyle{siam}
\bibliography{formal.bib}

\end{document}